\documentclass[10pt,reqno]{amsart}

\usepackage{amsmath,amsfonts,bbm,amsbsy,amsgen,amscd,mathrsfs,amssymb,amsthm}
\usepackage{enumerate,mathtools}

\usepackage[usenames,dvipsnames]{xcolor}
\usepackage[colorlinks=true,citecolor=blue,linkcolor=blue]{hyperref}

\usepackage{tikz}
\usetikzlibrary{shadings}

\def\bbone{{\mathbbm 1}}

\theoremstyle{plain}
\newtheorem{theorem}{Theorem}[section]
\newtheorem*{theorema}{Theorem A}
\newtheorem{lemma}[theorem]{Lemma}

\newtheorem{corollary}[theorem]{Corollary}
\newtheorem*{key estimate}{Key estimate}

\theoremstyle{definition}

\newtheorem*{key example}{Key example}

\theoremstyle{remark}
\newtheorem{remark}{Remark}

\DeclareMathOperator{\supp}{supp}

\newcommand{\R}{\mathbb{R}}

\numberwithin{equation}{section} 
\numberwithin{figure}{section}
\numberwithin{table}{section}

\begin{document}

\title{On sharp isoperimetric inequalities on the hypercube}

\author{David Beltran}
\address{(D.B.) Departament d’An\`alisi Matem\`atica, Universitat de Val\`encia, Dr. Moliner 50, 46100 Burjassot, Spain}
\email{david.beltran@uv.es}

\author{Paata Ivanisvili}
\address{(P.I.) Department of Mathematics, University of California, 
Irvine, CA 92617, USA}
\email{pivanisv@uci.edu}

\author{Jos\'e Madrid}
\address{(J.M.) Universite de Geneve, Geneva 4, Switzerland}
\email{Jose.Madrid@unige.ch}

\begin{abstract}
We prove the sharp isoperimetric inequality
$$
\mathbb{E} \,h_{A}^{\log_{2}(3/2)} \geq \mu(A)^{*} (\log_{2}(1/\mu(A)^{*}))^{\log_{2}(3/2)}
$$
for all sets $A \subseteq \{0,1\}^n$, where $\mu$ denotes the uniform probability measure, $\mu(A)^{*}=\min\{\mu(A), 1-\mu(A)\}$,  $h_A$ is supported on $A$ and to each vertex $x$ assigns  the number of neighbour vertices  in the complement of $A$. The inequality becomes equality for any subcube. 
Moreover, we provide lower bounds on $\mathbb{E} h_{A}^{\beta}$ in terms of $\mu(A)$ for all $\beta \in [1/2,1]$, improving, and in some cases  tightening,  previously known results. In particular, we obtain the sharp inequality   $\mathbb{E}h_{A}^{0.53}\geq 2 \mu(A)(1-\mu(A))$ for all sets with $\mu(A)\geq 1/2$, which allows us to refine a recent result of Kahn and Park on isoperimetric inequalities about partitioning the hypercube. 
Furthermore, we derive
Talagrand's isoperimetric inequalities for functions with values in a Banach space having finite cotype: for all $f :\{-1,1\}^{n} \to X$, $\|f\|_{\infty}\leq 1$, and any $p \in [1,2]$ we have 
 $$
 \|Df\|_{p} \gtrsim \frac{1}{q^{3/2}C_{q}(X)} \|f\|_{2}^{2/p}\left(\log \frac{e\|f\|_{2}}{\|f\|_{1}}\right)^{1/q},
 $$
 where $\| Df\|_{p}^{p} = \mathbb{E} \| \sum_{1\leq j \leq n} x'_{j} D_{j} f(x)\|^{p}$,  $x'$ is independent copy of $x$, and $C_{q}(X)$ is the cotype $q$ constant of $X$. Different proofs of the recently resolved Talagrand's conjecture will be presented. 

\end{abstract}



\maketitle

\thispagestyle{empty}

\section{Introduction}

Let $n \geq 1$ be an integer and $\{0,1\}^{n}$ be the hypercube of dimension $n$. One can regard $\{0,1\}^{n}$ as a graph where two vertices $x,y \in \{0,1\}^{n}$ are joined by an edge if the vectors $x=(x_{1}, \ldots, x_{n})$ and $y=(y_{1}, \ldots, y_{n})$ differ in exactly one coordinate. We denote such an edge by $(x,y)$. The goal of this article is to prove several sharp isoperimetric inequalities for subsets of $\{0,1\}^n$.

\subsection{One sided boundary}\label{sec:one-sided}
Let $A \subset \{0,1\}^{n}$ and $A^c:=\{0,1\}^n \backslash A$. Define the edge boundary of $A$ by $\nabla A := \{(x,y) : x \in A , y \in A^c\}$. Associated to $\nabla A$, we define the function $h_A: \{0,1\}^n \to \R$ by $h_{A}(x)=0$ if $x \in A^{c}$, and if $x \in A$ we let $h_{A}(x)$ be the number of edges joining $x$ with a vertex in $A^{c}$. Let  $\mu$ be the uniform probability measure on $\{0,1\}^n$ and let $\mathbb{E}$ denote the expectation operator. 
One of the goals of this paper is to investigate the quantity 
\begin{align}\label{conj1}
B(t,\beta,n) = \min_{A \subset \{0,1\}^{n}, \mu(A)=t}\mathbb{E} \, h_{A}^{\beta},
\end{align}
for any $\beta >0$, $t \in [0,1]$, $n \geq 1$; note that for fixed $n \geq 1$, the admissible values for $t \in [0,1]$ are of the form $t=m/2^n$ for all integers $0 \leq m \leq 2^n$.   

When $\beta=1$, the quantity $\mathbb{E} h_{A}$ has an additional structure, i.e., $\mathbb{E} h_{A} = \mathbb{E} h_{A^{c}}$, and we have $\mathbb{E}h_A=\frac{|\nabla A|}{2^n}$, where $|\cdot|$ denotes the cardinality of a set. In this case, it is known \cite{Hart} that
\begin{align}\label{harp1}
B(t,1,n) = n t-\frac{1}{2^{n-1}}\sum_{j=1}^{2^{n}t-1}s(j),
\end{align}
where $s(j)$ denotes the sum of ones in the binary representation of the integer $j$. The minimum in (\ref{conj1}) is achieved on the set $A$, with $\mu(A)=t$, consisting of those vertices $x=(x_{1},\ldots, x_{n}) \in \{0,1\}^{n}$ such that $\sum_{j=1}^{n} 2^{j-1} x_{j}=\ell$, for each $\ell=1,2, \ldots, t2^{n}$. The corresponding edge isoperimetric inequality derived from (\ref{harp1}) has applications in game theory \cite{Hart}, distributed algorithms, communication complexity and network science (see \cite{Cyrus} and references therein).

Despite the optimality of the function $B(t,1,n)$, it is often more convenient to use the lower bound  

\begin{equation}\label{eq:classical iso}
\mathbb{E} h_A \geq \mu(A)^{*} \log_2\Big(\frac{1}{\mu(A)^{*}}\Big),
\end{equation}
where $t^{*} = \min\{t,1-t\}$ for all $t \in [0,1]$. This is commonly known as the classical isoperimetric inequality on the hypercube. 
An advantage of the inequality $B(t,1,n) \geq t^{*} \log_{2}(1/t^{*})$ is that a) the right-hand side is independent of $n$; b) it becomes an equality\footnote{The inequality \eqref{eq:classical iso} also becomes an equality when  $\mu(A)=1-2^{-k}$ due to the identity $\mathbb{E}h_{A} = \mathbb{E}h_{A^{c}}$} at the points $t=2^{-k}$ for $0\leq k \leq n$. In this case ($\beta=1$, $t=2^{-k}$),  the minimum in (\ref{conj1}) is achieved on subcubes of co-dimension $k$, for any $0\leq k \leq n$.  In what follows we will be interested in  $\inf_{n\geq 1} B(t,\beta,n)$. 

By considering hamming balls  $\{x_{1}+\ldots+x_{n} \leq [n/2]\}$ in (\ref{conj1}) one can see that 
\begin{equation}\label{eq:beta>=1/2}
\text{if} \quad \beta<1/2 \quad \text{then} \quad \inf_{n \geq 1} B(t,\beta,n)=0;
\end{equation}
see Section \ref{sec:nec}. On the other hand, by considering subcubes  in (\ref{conj1}) it follows that  
\begin{align}\label{sharp1}
t (\log_{2}(1/t))^{\beta} \geq \inf_{n \geq 1} B(t,\beta,n) \quad \text{at the points} \quad  t=2^{-k},  \quad k=0,1,2, \ldots.
\end{align}
A straightforward application of the H\"older inequality in \eqref{eq:classical iso} implies
\begin{equation*}
\mathbb{E} h_{A}^\beta  \geq \mu(A) \Big( \log_2\Big(\frac{1}{\mu(A)}\Big)\Big)^\beta \qquad \text{for all $\beta \geq 1$},
\end{equation*}
and combining this with \eqref{sharp1} establishes $\inf_{n \geq 1} B(t,\beta, n)=t (\log_{2}(1/t))^{\beta}$ at   $t=2^{-k}$ for all $k \geq 0$ and all $\beta \geq 1$. From now on, we will be concerned with $1/2 \leq \beta < 1$.

The case $\beta=1/2$ was first considered by Talagrand in  \cite{Tal93}. By an inductive argument, he proved the reverse inequality to (\ref{sharp1}) up to a universal constant $K>1$, namely
\begin{align}\label{tt1}
\mathbb{E} \sqrt{h_{A}} \geq  K \mu(A)^{*} \sqrt{\log_{2}(1/\mu(A)^{*})}
\end{align}
for all $A \subseteq \{0,1\}^n$. 
Furthermore, he also proved that
\begin{align}\label{tt2}
\mathbb{E} \sqrt{h_{A}} \geq \sqrt{2}\mu(A)(1-\mu(A)) 
\end{align}
holds for all $A \subset \{0,1\}^{n}$. Bobkov and G\"otze \cite{bob1} came up with an improved and elegant inductive argument and, in particular, one of their results implies that $\sqrt{2}$ can be replaced by $\sqrt{3}$ in (\ref{tt2}). Recently, Kahn and Park \cite{KP} showed that 
\begin{align}\label{kpp1}
\mathbb{E} h^{\log_{2}(3/2)}_{A} \geq 2 \mu(A)(1-\mu(A)), 
\end{align}
and the constant $2$ in the right-hand side of (\ref{kpp1}) is sharp: the inequality becomes an equality for subcubes of co-dimension 1 and 2. In fact, for any given $\beta \geq 1/2$, the largest possible constant $C_\beta$ in an inequality of the form $\mathbb{E} h_A^{\beta} \geq C_\beta \mu(A) (1-\mu(A))$ for all $A \subseteq \{0,1\}^n$ is $C_\beta \leq 2$; this follows from testing the inequality on a half cube. Furthermore, by considering subcubes of co-dimension 2, the constant $C_\beta=2$ is only possible if $\beta \geq \log_2(3/2)$.

In this paper we prove a theorem which, in particular, reverses the inequality (\ref{sharp1}) for  all $\beta \geq  \log_{2}(3/2)=0.5849...$ and therefore gives the value for $\inf\limits_{n\geq 1} B(t, \beta, n)$  for that range of $\beta$ whenever $t=2^{-k}$, $k\geq 0$.  

\begin{theorem}\label{mth11}
Let $\beta_0:=\log_2(3/2)$. Then the inequality 
\begin{align}\label{iso2}
\mathbb{E} h^{\beta_{0}}_{A} \geq \mu(A)^{*} (\log_{2}(1/\mu(A)^{*}))^{\beta_{0}} 
\end{align}
holds for all $A \subset \{0,1\}^{n}$. In particular, if $\mu(A)\leq 1/2$ we have 
\begin{align}\label{iso1}
\mathbb{E} h^{\beta}_{A} \geq \mu(A) (\log_{2}(1/\mu(A)))^{\beta} \qquad \text{for all} \quad \beta \geq \log_{2}(3/2).
\end{align}
 The equality holds in both  (\ref{iso2}) and (\ref{iso1}) for any subcube $A \subset \{0,1\}^{n}$. 
\end{theorem}

Note that the above theorem is a tight version of \eqref{kpp1} on all subcubes. The comparison between the inequalities (\ref{kpp1}) and (\ref{iso2}) is that the function $2t(1-t)$ is larger or equal than $t^{*}(\log_{2} (1/t^{*}))^{\beta_{0}}$ on $[1/4,3/4]$, whereas the second function is larger or equal than the first one in the complementary range. In general, a similar phenomenon happens when comparing $2t(1-t)$ and $t^{*}(\log_{2} (1/t^{*}))^{\beta}$ for other values of $1/2 \leq \beta \leq 1$: the second one is significantly larger in neighborhoods of $t=0$ and $t=1$ and, moreover, it satisfies (\ref{sharp1}). It is an interesting question to understand whether the inequality (\ref{iso2}) holds for all $\beta_{0}  \geq 1/2$. 

Let $\partial A = \mathrm{supp}(h_{A})\subset A$ be the vertex boundary of $A$. An application of the H\"older inequality $\mathbb{E} h^{\beta}_{A} \leq (\mathbb{E}h_{A}^{\beta/\alpha})^{\alpha} \mu(\partial A)^{1-\alpha}$ for all $\alpha \in [0,1]$ implies the following corollary, which sharpens the classical isoperimetric inequality \eqref{eq:classical iso}.
\begin{corollary}
For $\gamma= \frac{1}{\log_{2}(3/2)}-1=0.709...$, and any $A \subset \{0,1\}^{n}$ with $\mu(A)\leq 1/2$ we have 
\begin{align}\label{uncertain}
\mathbb{E} h_{A} \geq  \left(\frac{\mu(A)}{\mu(\partial A)}\right)^{\gamma} \mu(A) \log_{2}\Big(\frac{1}{\mu(A)}\Big).
\end{align}
Moreover, the equality holds for any subcube. 
\end{corollary}

The constant $\sqrt{3}$ in Bobkov's inequality $\mathbb{E} \sqrt{h_{A}} \geq \sqrt{3} \mu(A)(1-\mu(A))$ is not sharp. We obtain the following improvement.

\begin{theorem}\label{thm:quadratic}
For any $\beta \in [1/2, \log_{2}(3/2)]$ we have 
\begin{align}\label{iso90}
\mathbb{E} h_{A}^{\beta}\geq C_\beta \, \mu(A)(1-\mu(A))
\end{align}
for all $A \subset \{0,1\}^{n}$, where $C_\beta = 2 \sqrt{2^{\beta+1}-2}$. 
\end{theorem}
The theorem applied to the case $\beta=1/2$ gives the improved constant $C_{1/2}=1.82...>\sqrt{3} =1.73... .$
Testing the inequality $\mathbb{E}h_A^\beta \geq C_\beta \mu(A)(1-\mu(A))$ on subcubes of co-dimension 2, we obtain the upper bound  $C_\beta \leq 2^{\beta+2}/3$ for $1/2 \leq \beta \leq \log_2(3/2)$. For $\beta = \log_{2}(3/2)$ the inequality (\ref{iso90}) coincides with (\ref{kpp1}).

\subsection{Separating the cube}\label{intro:sep}

If one considers sets $A \subset \{0,1\}^n$ with measure $\mu(A) \geq 1/2$, it is conceivable that $\mathbb{E} h_A^\beta \geq 2 \mu(A)(1-\mu(A))$ could hold for a wider range of exponents $\beta \geq 1/2$. In this direction we obtain the following.

\begin{theorem}\label{thm:cubic}
For all $A\subset\{0,1\}^n$,
\begin{align}\label{eq: cubic}
    \mathbb{E} h_A^{0.53}   \geq 8 \mu(A) (1-\mu(A)) \Big[\Big(1-\frac{2 \sqrt{2}}{3}\Big)\mu(A) + \frac{\sqrt{2}}{3} - \frac{1}{4}\Big].
\end{align}
\end{theorem}

By an immediate comparison of the right-hand side of \eqref{eq: cubic} with $2 \mu(A)(1-\mu(A))$ we obtain an extension of (\ref{kpp1}).
\begin{corollary}\label{cor:quad large measure}
If $\mu(A)\geq 1/2$ then
 \begin{align*}
 \mathbb{E} h_A^{0.53}   \geq 2 \mu(A) (1-\mu(A)).  
 \end{align*}
\end{corollary}

The above Corollary has an immediate application in isoperimetric inequalities involving partitions of $\{0,1\}^n$. Let  $(A,B,W)$ be a partition of $\{0,1\}^n$, with $W$ typically thought to be small. With this setup, define
$\nabla(A,B)=\{(x,y):x \in A, y \in B\}.$ For $\beta \geq 1/2$, Kahn and Park~\cite{KP} considered inequalities of the form 
\begin{equation}\label{eq:conj n beta}
    |\nabla(A,B)|+K n^\beta  |W| \ge 2^{n-1}
\end{equation}
for sets $A$ with measure $\mu(A)=1/2$ and some absolute constant $K>0$. Note that if $W=\emptyset$ (that is, $B=A^c$), then \eqref{eq:conj n beta} becomes $|\nabla A| \geq |A|$, which is the classical isoperimetric inequality \eqref{eq:classical iso} for sets of measure $\mu(A)=1/2$. 
They conjectured \cite[Conjecture 1.3]{KP} that \eqref{eq:conj n beta} should hold for $\beta=1/2$; here $\beta \geq 1/2$ is again necessary from testing the inequality on Hamming balls (see Section \ref{sec:nec}). 
By an elementary argument (see the end of \S\ref{sec:cubic} for completeness), they showed that the inequality \eqref{kpp1} implies \eqref{eq:conj n beta} for $\beta=\log_2(3/2)$ and $K=1$. Our Corollary \ref{cor:quad large measure} yields a further improvement in the exponent $\beta$.

\begin{corollary}\label{cor:sep cube}
Let $(A,B,W)$ be a partition of $\{0,1\}^n$ such that $\mu(A)=1/2$. Then
\[ 
|\nabla(A,B)|+n^{0.53}|W| \ge 2^{n-1}.\]
\end{corollary}

\medskip

\subsection{Two-sided boundary}\label{intro:two-sided}
Given $A \subseteq \{0,1\}^n$, define the two-sided boundary function $w_A:\{0,1\}^n \to \R$ by $w_A(x)=h_A(x)$ if $x \in A$ and $w_A(x)=h_{A^c}(x)$ if $x \in A^c$. We remark that one can obtain bounds for $\mathbb{E} w_A^\beta$ from those for $\mathbb{E} h_A^\beta$ noting that $w_A^\beta = h_A^\beta + h_{A^c}^\beta$. Indeed, if there is a function $B:[0,1] \to \R$ such that $\mathbb{E}h_A^\beta \geq B(\mu(A))$ holds for all $A \subseteq \{0,1\}^n$, then $\mathbb{E}w_A^\beta  \geq B(\mu(A)) + B(1-\mu(A)).$
In particular, using Theorem \ref{mth11} one obtains
\[
\inf_{n \geq 1, \mu(A) = t} \mathbb{E}w_A^{\beta_0} \geq 2 t^{*} (\log_2(1/t^{*}))^{\beta_0}, \qquad \beta_0=\log_2(3/2),
\]
and combining Theorem \ref{mth11} with Theorem \ref{thm:cubic} one has
\[
\inf_{n \geq 1, \mu(A) = t} \mathbb{E}w_A^{\beta} \geq t^* (\log_2(1/t^*))^{\beta}  +  P(t_*)\qquad \text{ for } \,\,\beta > \log_2(3/2),
\]
where $t_*:=\max(t,1-t)$ and $P$ is the cubic polynomial on the right-hand side of \eqref{eq: cubic}. In the particular case $\mu(A)=1/2$ we obtain sharp lower bounds for all $\beta \geq 0.53$ using Corollary \ref{cor:quad large measure}. 
\begin{corollary}\label{kubiki}
For all $\beta \geq 0.53$,
\[
\inf_{n \geq 1, \mu(A) = 1/2} \mathbb{E}  w_A^{\beta} =  1.
\]
Moreover, for all $0.5 \leq \beta < 0.53$,
\begin{align}\label{twoss}
1 \geq \inf_{n \geq 1, \mu(A) = 1/2} \mathbb{E} w_A^{\beta} \geq   \sqrt{2^{\beta+1}-2}.
\end{align}
\end{corollary}
The upper bound $1 \geq \inf_{n \geq 1, \mu(A) = 1/2} \mathbb{E} w_A^{\beta}$  comes from testing on a half-cube, and the non-sharp lower bound in (\ref{twoss})  in the range $0.5 \leq \beta < 0.53$ follows from Theorem \ref{thm:quadratic}. 
This improves over a recent result in 
\cite[Section 3.6]{INV} which says that 
\begin{align*}
1\geq \inf_{n\geq 1, \; \mu(A)=1/2}\mathbb{E} w_{A}^{\beta} \geq \max 
\left\{ \sqrt{\frac{2}{\pi}}, s_{(2\beta)'}^{2\beta} \right\}\quad \quad \text{for  $\,\,1/2\leq \beta \leq 1$}.
\end{align*}
Here, for any $1 \leq q < \infty$, $q'=\frac{q}{q-1}$, and $s_{q}$ denotes the smallest positive zero of the confluent hypergeometric functions $x \mapsto\ _{1}F_{1}\left(-\frac{q}{2}, \frac{1}{2}, \frac{x^{2}}{2}\right)$. 
 We remark that $s_{p'}\geq \sqrt{\frac{2(p-1)}{p}}$ for all $p \in [1,2]$ and $s_{p'} \to 0$ as $p \to 1$.  Corollary~\ref{kubiki} for $\beta=1/2$ gives
 $$
 \inf_{n \geq 1, \mu(A) = 1/2} \mathbb{E} \sqrt{w_A}\geq \sqrt{2^{3/2}-2} =    0.91...>\sqrt{\frac{2}{\pi}}=0.79... .
 $$

\begin{remark}
Let $C_{1}$ be the largest constant in the $L^{1}$ Poincar\'e inequality 
\begin{align}\label{l1p}
C \mathbb{E} |f-\mathbb{E}f| \leq \mathbb{E} |\nabla f|
\end{align}
 for all $f: \{-1,1\}^{n} \mapsto \mathbb{R}$ and all $n\geq 1$ (see Section~\ref{taltal1} for the meaning of $|\nabla f|$). Let $C_{1,B}$ be the largest constant in (\ref{l1p}) for all $f :\{-1,1\}^{n} \mapsto \{0,1\}$. Clearly $C_{1}\leq C_{1,B}$. It was proved in \cite{PVR} that $\frac{2}{\pi}<C_{1} \leq \sqrt{\frac{2}{\pi}}$, and it is believed that $C_{1}=\sqrt{\frac{2}{\pi}}$. One may suspect that $C_{1} = C_{1,B}$. However,  notice that for an indicator function $f=1_{A}$, $A \subset \{-1,1\}^{n}$, we have 
 \begin{align*}
 \frac{\mathbb{E} |\nabla f|}{\mathbb{E} |f-\mathbb{E}f|} = \frac{\mathbb{E} \sqrt{w_{A}}}{4 \mu(A)(1-\mu(A))};
 \end{align*}
 see Section~\ref{taltal1} for the relation between $|\nabla f|$ and $w_{A}$. On the other hand, Theorem~\ref{thm:quadratic} combined with the fact $\sqrt{w_{A}} = \sqrt{h_{A}}+\sqrt{h_{A^{c}}}$ shows $\sqrt{2^{3/2}-2} \leq C_{1,B}$ which implies  $C_{1,B}-C_{1} \geq\sqrt{2^{3/2}-2}  - \sqrt{2/\pi} =  0.112...>0$. 
\end{remark}
\subsection{Talagrand's isoperimetric inequalities for the discrete gradient}\label{taltal1}

It will be convenient to work with $\{-1,1\}^{n}$ instead of $\{0,1\}^{n}$.\footnote{Note that $h_A$ and $w_A$ can be defined analogously for sets $A \subset \{-1,1\}^n$.}  For any $f:\{-1,1\}^{n} \mapsto \mathbb{R}$ and any $x = (x_{1},, \ldots, x_{n})\in \{-1,1\}^{n}$ we set 
\begin{align*}
D_{j}f(x) = \frac{f(x_{1}, \ldots, x_{j}, \ldots, x_{n})-f(x_{1}, \ldots, -x_{j}, \ldots, x_{n})}{2}.
\end{align*}
We say $f$ is boolean if $f$ takes values $0$ or $1$. Define 
\begin{align*}
|\nabla f(x)| = \sqrt{\sum_{j=1}^{n} |D_{j} f|^{2}}, \quad \text{and} \quad |M f(x)| = \sqrt{\sum_{j=1}^{n} (D_{j} f)_{+}^{2}},
\end{align*}
where $(a)_{+} = \max\{0,a\}$ for any $a \in \mathbb{R}$.
For any $A \subset \{-1,1\}^{n}$,  let  $1_{A}$  denote the indicator function of the set $A$.
Notice that if $x \notin A$ then $|M 1_{A}(x)|^{2}=0$, and if $x \in A$ then $|M 1_{A}(x)|^{2}$ equals  1/4 times   the number of neighbors of $x$ in $A^{c}$.  Thus $|M 1_{A}(x)|^{2} = h_{A}(x)/4$. On the other hand, the ``two-sided'' gradient $|\nabla 1_{A}|^{2}$ takes the form 
$|\nabla 1_{A}(x)|^{2} =h_{A}(x)/4 + h_{A^{c}}(x)/4=w_A(x)/4$. 

In this section we will be concerned with functional isoperimetric inequalities which are not sharp in terms of constants but capture the correct behaviour of the quantities involved. To this end, we introduce the notation $A \gtrsim B$ to mean that $A \geq C B$ with some universal constant $C>0$. 

 Any function $f :\{-1,1\}^{n} \to\mathbb{R}$ has a Fourier--Walsh representation 
 \begin{align*}
f(x) = \sum_{S \subset \{1,\ldots, n\}} \hat{f}(S) x^{S}, \quad x^{S} = \prod_{j \in S} x_{j}.
 \end{align*}
Let $p\geq 1$, and  define $\|f\|_{p} = (\mathbb{E} |f|^{p})^{1/p}$. Recently, using random restriction arguments in an elegant way,  Eldan--Kindler--Lifshitz--Minzer obtained the following.
\begin{theorema}[\cite{Dor}]\label{dor11}
For any $f : \{-1,1\}^{n} \mapsto \{-1,1\}$, and any $p \in [1,2]$,  we have 
\begin{align}\label{fbound}
\|\nabla f\|_{p} \gtrsim \sup_{d\geq 0} \left(\sum_{|S|\geq d} |\hat{f}(S)|^{2}\right)^{1/p} \sqrt{d}.
\end{align}
\end{theorema}
It was also explained in \cite{Dor} that essentially by choosing $d  \approx \log \frac{e}{\mathrm{Var}(f)}$,  the inequality (\ref{fbound}) implies Talagrand's  sharp isoperimetric inequality,
\begin{align}\label{tala11}
\|\nabla f\|_{p} \gtrsim (\mathrm{Var}(f))^{1/p} \sqrt{\log \frac{e}{\mathrm{Var}(f)}}
\end{align}
for all $f : \{-1,1\}^{n} \mapsto \{-1,1\}$, and by choosing $d \approx \log \frac{e}{W(f)}$, where $W(f) = \sum_{j=1}^{n} (\mathbb{E} |D_{j}f|)^{2}$,  the inequality (\ref{fbound})  implies  Talagrand's conjecture\footnote{The case $p \in (1,2]$ was known to Talagrand, and he  was asking the question  in the endpoint case $p=1$.}
\begin{align}\label{tac1}
\|\nabla f\|_{p} \gtrsim (\mathrm{Var}(f))^{1/p}  \sqrt{\log \frac{e}{W(f)}}
\end{align}
for all $f : \{-1,1\}^{n} \mapsto \{-1,1\}$, which was recently resolved in \cite{ELGR}. Notice that for {\em monotone} boolean functions $f$, i.e., $f(x) \geq f(y)$ whenever $x_{i}\geq y_{i}$ for all $i=1,\ldots, n$, we have 
$$
W(f) = \sum_{j=1}^{n} \hat{f}(\{j\})^{2} \leq \sum_{S \neq \emptyset} \hat{f}(S)^{2} = \mathrm{Var}(f).
$$

\begin{remark}
We should point out that the inequality (\ref{tac1}) for $p=1$ can be obtained from Theorem 31 in \cite{KDS}, which states that 
$\mathbb{NS}_{\delta}(f) \lesssim \sqrt{\delta}\,  \mathbb{E} |\nabla f|$ holds for all $\delta \in [0,1]$ and all $f :\{-1,1\}^{n} \to \{-1,1\}$, where $\mathbb{NS}_{\delta}(f)$ is the {\em noise sensitivity} of $f$.  Using the bound  $\mathbb{NS}_{\delta}(f) \gtrsim \mathrm{Var}(f) (1-\left(W(f)/\mathrm{Var}(f)\right)^{\frac{\delta}{2-\delta}})$ from \cite{kindler}, and choosing $\delta = \frac{1}{\log(\mathrm{Var}(f)/W(f))}$, where we can  assume $\mathrm{Var}(f)/W(f)>100$ by (\ref{tala11}), we obtain 
$\mathbb{E} |\nabla f| \gtrsim \mathrm{Var}(f) \sqrt{\log \frac{\mathrm{Var}(f)}{W(f)}}$. Invoking again (\ref{tala11}) we can assume $\mathrm{Var}(f) \geq \sqrt{W(f)}$ which leads us to one more proof of  (\ref{tac1}).
\end{remark}

In this paper we present new arguments in obtaining (\ref{fbound}) which unlike random restriction methods \cite{Dor} or inductive proofs \cite{KDS} extend the results to functions with values in an arbitrary Banach space having finite cotype. Recall that a Banach space $(X,\|\cdot \|)$ has Rademacher cotype $q \in [2,\infty]$ if  there exists a finite $C$ such that
\begin{align*}
 \sum_{j=1}^{n}\|x_{j}\|^{q} \leq C^{q} \mathbb{E} \|\sum_{j=1}^{n} \varepsilon_{j} x_{j} \|^{q}
\end{align*}
holds for all $n\geq 1$,  and all  $x_{1},\ldots, x_{n} \in X$, where $\varepsilon_{1}, \ldots, \varepsilon_{n},$ are independent identically distributed symmetric $\pm1$ Bernoulli random variables. The best constant $C=C_{q}(X)$ is called the cotype $q$ constant of $X$.   
 Let $\Delta  = \sum_{j=1}^{n} D_{j}$, and let $e^{-t\Delta}$ be the heat semigroup.  
\begin{theorem}\label{vectortal}
Let $f:\{-1,1\}^{n} \to X$, $\|f\|_{\infty}\leq 1$. Then for any $p\in [1,2]$ we have 
\begin{align*}
 \| Df\|_{p}\gtrsim  \frac{1}{q^{3/2}C_{q}(X)}\;  \|f\|_{2}^{2/p} \left(\log  \frac{e\|f\|_{2}}{\|f\|_{1}} \right)^{1/q},  
\end{align*}
where $\| Df\|_{p}^{p} = \mathbb{E} \| \sum_{1\leq j \leq n} x'_{j} D_{j} f(x)\|^{p}$ and $x'$ is independent copy of $x$.
\end{theorem}

\begin{remark}
If $f$ is boolean, then one can apply the theorem to $f-\mathbb{E}f$ and use Khintchine's inequality to recover Talagrand's isoperimetric inequality (\ref{tala11}).
\end{remark}

Theorem~\ref{vectortal} will be obtained as a consequence of the following ``noise sensitivity'' inequality for vector valued functions. 
\begin{theorem}\label{nstabilityt}
For all $p\in [1,2]$ the inequality
\begin{align}\label{nstability}
\|f-e^{-t\Delta}f\|_{p}   \lesssim C_{q}(X) q^{3/2} (1-e^{-2t})^{\frac{1}{q}} \|Df\|_{p}, \quad t\geq 0,
\end{align}
holds for all  $f:\{-1,1\}^{n} \to X$. 
\end{theorem}

\begin{remark}
If $f$ is boolean, then Theorem~\ref{nstabilityt} together with the simple estimate $\|f-e^{-\Delta/d}f\|_{p} \gtrsim \left( \sum_{|S|\geq d} \hat{f}(S)^{2}\right)^{1/p}$, $p \in [1,2]$,  implies (\ref{fbound}).
\end{remark}

\bigskip

\noindent \textit{Acknowledgments}. D.B. was partially supported by the NSF grant DMS-1954479 and the AEI grant RYC2020-029151-I.  P.I was supported in part by  NSF grant CAREER-DMS-2152401.

\section{The proofs}
\subsection{An inductive lemma}

The next lemma is a standard inductive argument which  follows the same steps as in \cite{Tal93, bob1} with the latest upgrade found in \cite{KP}. 

\begin{lemma}[{\cite[pp. 4217-4219]{KP}}]\label{th44}
    Let $\beta \in [1/2,1]$. If $B : [0,1] \to [0,\infty)$ satisfies $B(0)=B(1)=0$, $B(1/2)\leq 1/2$, and the inequality
\begin{align}
&\max\{((y-x)^{1/\beta}+B(y)^{1/\beta})^{\beta}, y-x+(2^{\beta}-1)B(y)\} \nonumber \\
&+B(x)\geq 2B\Big(\frac{y+x}{2}\Big) \label{sami}
\end{align}
    holds for all $0 \leq x \leq y \leq 1$, then $\mathbb{E}h_{A}^{\beta}\geq B(\mu(A))$ for all $A\subset \{0,1\}^{n}$ and all $n\geq 1$. 
\end{lemma}

This Lemma will be used in the proofs of  Theorems \ref{mth11}, \ref{thm:quadratic} and \ref{thm:cubic}.

\begin{remark}
Notice that $\max\{((y-x)^{1/\beta}+B(y)^{1/\beta})^{\beta}, y-x+(2^{\beta}-1)B(y)\} = ((y-x)^{1/\beta}+B(y)^{1/\beta})^{\beta}$ if and only if $B(y)\geq y-x$. 
\end{remark}
\subsection{Proof of Theorem~\ref{mth11}: logarithmic function}

The following ``two-point inequality'' will play a crucial role in our proofs. 

\begin{lemma}\label{help1}
For any $\beta \in  [\frac{1}{2},1]$ and all $0\leq x\leq y \leq \frac{1}{2}$ we have 
\begin{align}\label{twop99}
((y-x)^{1/\beta}+B(y)^{1/\beta})^{\beta} +B(x)\geq 2B\Big(\frac{y+x}{2}\Big),
\end{align}
where $B(x) = x \big( \log_{2}(1/x)\big)^{\beta}$.
\end{lemma}
\begin{proof}
First we rewrite the inequality (\ref{twop99}) as 
\begin{align}\label{twop}
\left(y^{1/\beta} \log_{2}(1/y) +(y-x)^{1/\beta} \right)^{\beta}+x (\log_{2}(1/x))^{\beta} \geq (x+y)\Big( \log_{2}\Big(\frac{2}{x+y}\Big)\Big)^{\beta}.
\end{align}
We divide the proof into three steps.
\medskip

\noindent \underline{Step 1}. If the inequality holds for $\beta=1/2$, then it holds for all $\beta \geq 1/2$. Indeed, let us rewrite the inequality \eqref{twop} as
\begin{align}\label{twop1}
\left[\frac{y}{x+y} \Big(\log_{2}(1/y)+\Big(\frac{y-x}{y}\Big)^{1/\beta} \Big)^{\beta} + \frac{x}{x+y}(\log_{2}(1/x))^{\beta}\right]^{1/\beta}\geq  \log_{2}\left(\frac{2}{x+y}\right).
\end{align}
Next, apply   $\Big(\frac{y-x}{y}\Big)^{1/\beta} \geq \Big(\frac{y-x}{y}\Big)^{2}$ and the fact that $\beta \mapsto (\mathbb{E} |X|^{\beta})^{1/\beta}$ is non-decreasing on $[0,\infty)$ for a random variable $X$. 

\medskip

In what follows we assume $\beta=\frac{1}{2}$ in (\ref{twop}).

\medskip

\noindent \underline{Step 2}. The inequality for $y=\frac{1}{2}$ implies the inequality for all $y \in [x,\frac{1}{2}]$. 
Indeed, let $x=yt$ where $0\leq t \leq 1$,   and let $a = \log_{2}(1/y) \in [1,\infty)$. Then (\ref{twop1}) can be rewritten as 
\begin{align*}
\left[\frac{1}{t+1} \Big(a+(1-t)^{1/\beta} \Big)^{\beta} + \frac{t}{t+1}(a+\log_{2}(1/t))^{\beta}\right]^{1/\beta}-a\geq  \log_{2}\left(\frac{2}{t+1}\right).
\end{align*}
Notice that $\varphi(a) = (\mathbb{E}(a+|X|)^{\beta})^{1/\beta}-a$ is non-decreasing on $(0,\infty)$ for a random variable $X$.  Indeed, 
\begin{align*}
\varphi'(a) = (\mathbb{E} (a+|X|)^{\beta})^{\frac{1}{\beta}-1}\mathbb{E} (a+|X|)^{\beta-1}-1 \geq 0
\end{align*}
by H\"older's inequality. 

\medskip

Hence it suffices to verify the inequality for $y=1/2$. 
\medskip

\noindent \underline{Step 3}. We verify the inequality
\begin{align}\label{22point}
 \sqrt{1+(1-t)^{2}} + t\sqrt{1+\log_{2}(1/t)} \geq (t+1)\sqrt{1+ \log_{2}\left(\frac{2}{t+1}\right)}, \quad t \in[0,1]. 
\end{align}
If we denote $1-t=s$, then the inequality \eqref{22point} takes the form 
\begin{align}\label{sqr}
\sqrt{1+s^{2}}+(1-s)\sqrt{1-\log_{2}(1-s)}\geq (2-s)\sqrt{1+\log_{2}\left(\frac{2}{2-s}\right)}.
\end{align}
Next we use two simple inequalities 
\begin{align}\label{logar1}
&-\log_{2}(1-s) \geq \frac{s}{\ln(2)}+\frac{s^{2}}{2\ln(2)};\\
&\log_{2}\left(\frac{2}{2-s}\right)\leq \frac{s}{2\ln(2)}+\left( 1-\frac{1}{2\ln(2)}\right)s^{2}. \label{logar2}
\end{align}
The first inequality follows by Taylor's theorem. The second inequality follows by considering the function $g(s) = \frac{s}{2\ln(2)}+\left( 1-\frac{1}{2\ln(2)}\right)s^{2} - \log_{2}\left(\frac{2}{2-s}\right)$. Note that $g(0)=g(1)=0$. Furthermore, 
\begin{align*}
\frac{g'(s)}{s} = - \frac{(4\ln(2)-2)s+5-8\ln(2)}{2(2-s)\ln(2)}
\end{align*}
changes sign only once from $+$ to $-$ on $(0,1)$ (notice that $5-8\ln(2)<0$ and $3-4\ln(2)>0$). Hence $g(s) \geq 0$ on $[0,1]$. 

Next, we replace the logarithms in \eqref{sqr} by their lower and upper bounds correspondingly obtained in \eqref{logar1} and \eqref{logar2}. After squaring both sides of the obtained inequality, it suffices to prove
\begin{align*}
&1+s^{2}+2(1-s)\sqrt{(1+s^{2})\left(1+ \frac{s}{\ln(2)}+\frac{s^{2}}{2\ln(2)}\right)}+(1-s)^{2}\left(1+ \frac{s}{\ln(2)}+\frac{s^{2}}{2\ln(2)}\right)\\
&\geq (2-s)^{2}\left(1+\frac{s}{2\ln(2)}+\left( 1-\frac{1}{2\ln(2)}\right)s^{2} \right). 
\end{align*}
Isolating the square root term on the left-hand side of the inequality, and squaring both sides again, it suffices to prove
\begin{align*}
&4(1-s)^{2}(1+s^{2})\left(1+ \frac{s}{\ln(2)}+\frac{s^{2}}{2\ln(2)}\right) \geq \left[ -(1-s)^{2}\left(1+ \frac{s}{\ln(2)}+\frac{s^{2}}{2\ln(2)}\right) \right.\\
&\left. + (2-s)^{2}\left(1+\frac{s}{2\ln(2)}+\left( 1-\frac{1}{2\ln(2)}\right)s^{2} \right) - 1-s^{2} \right]^{2}.
\end{align*}
After opening the parentheses we see that the difference (the left-hand side minus the right-hand side) simplifies to $-\frac{s^{2}(1-s)^{2}}{4\ln^{2}(2)}\times v(s)$, where 
\begin{align*}
&v(s) = 4 (\ln(2)-1)^{2} s^{4}+12(1-\ln(2))(2\ln(2)-1)s^{3}-(17-18\ln(2))(2\ln(2)-1)s^{2}\\
&-(12-24\ln(2)+16\ln^{2}(2))s+32\ln^{2}(2)-32\ln(2)+4.
\end{align*}
The factor $s^{2}(1-s)^{2}$ should not be surprising as $s=0$ and $s=1$ are the equality cases in (\ref{sqr}) and, by the several reductions that we did, we kept the values at $s=0$ and $s=1$ unchanged. 
The signs of the coefficients in $v(s)$, from the highest to the lowest power,  are $+,+,-,-,-$. Therefore, by Descartes' rule of sign,  the number of positive roots of the polynomial $v$ is at most one. As $v(s)$ is positive at infinity, and $v(1) = 32\ln^{2}(2)-32\ln(2)+1<0$ it follows that $v(s)<0$ on $[0,1]$, and thus $g(s)>0$ on $[0,1]$, concluding the proof of the inequality (\ref{22point}) and hence of Lemma~\ref{help1}.
\end{proof}

The following lemma will be used for sets of measure greater than $1/2$.

\begin{lemma}\label{help2}
For any $\beta \in  [\frac{1}{2},\ln(2)]$ and all $0\leq x\leq y \leq \frac{1}{2}$ we have 
\begin{align}\label{twop999}
\big((y-x)^{1/\beta}+B(x)^{1/\beta}\big)^{\beta} +B(y)\geq 2B\Big(\frac{y+x}{2}\Big),
\end{align}
where $B(x) = x \left( \log_{2}(1/x)\right)^{\beta}$.
\end{lemma}
\begin{proof}
Notice that $B$ is non-decreasing on $[0,1/2]$ since
\begin{align*}
B'(x) 
& = (\log_{2}(1/x))^{\beta-1}\Big(\log_{2}(1/x) - \frac{\beta}{\ln(2)} \Big)\geq 0
\end{align*}
for $\beta \in [1/2,\ln(2)]$. 
For all $0\leq x \leq y \leq 1/2$ we have 
\begin{align}\label{jose1}
((y-x)^{1/\beta}+B(x)^{1/\beta})^{\beta} - ((y-x)^{1/\beta}+B(y)^{1/\beta})^{\beta}+B(y)-B(x)\geq 0.
\end{align}
Indeed, for each $\beta \in (0,1]$  the map $\psi(t) =((y-x)^{1/\beta}+t^{1/\beta})^{\beta}-t$ is non-increasing on  $[0, \infty)$. Thus, using the fact $B(x)\leq B(y)$ it follows that $\psi(B(x)) \geq \psi(B(y))$, which is the same as (\ref{jose1}). Finally, the left-hand side of   (\ref{twop999}) minus the  left-hand side of (\ref{jose1}) coincides with the left-hand side of (\ref{twop99}) and the lemma follows.
\end{proof}

The next lemma is the main technical part of  \cite{KP}. We note that this result is also contained in Lemmas \ref{lemma:quad case1} and \ref{lemma:quad case 2} for $\beta=\log_2(3/2)$.

\begin{lemma}[{\cite[pp. 4219-4220]{KP}}]\label{help3}
The function $B(x)=2x(1-x)$ satisfies \eqref{sami} for  $\beta = \log_{2}(3/2)$. 
\end{lemma}

Now we are ready to prove Theorem~\ref{mth11}. Let $\beta = \beta_{0} = \log_{2}(3/2)$ in Lemma~\ref{th44}, and consider 
\begin{align}\label{eq:B}
B(t) := 
\begin{cases}
t (\log_{2}(1/t))^{\beta_{0}} & t \in [0,1/4),\\
2t(1-t) & t \in [1/4,3/4],\\
(1-t) (\log_{2}(1/(1-t)))^{\beta_{0}} & t \in (3/4,1].\\
\end{cases}
\end{align}
Notice that 
\begin{align*}
B(t) = \max\{ t^{*} (\log_{2}(1/t^{*}))^{\beta_{0}}, 2t(1-t)\}=:\max\{B_{1}(t), B_{2}(t)\}, \quad t\in [0,1],
\end{align*}
where $t^{*} = \min\{t, 1-t\}$. The left-hand side of the inequality (\ref{sami}) can be rewritten as $\Phi(B(x),B(y),(y-x))$, where 
\begin{align*}
\Phi(u,v,w) := \max\{ (w^{1/\beta_0}+v^{1/\beta_0})^{\beta_0},w+(2^{\beta_0}-1)v\} + u.
\end{align*}
The map $(u,v,w) \in [0,\infty)^{3} \mapsto \Phi(u,v,w)$ is non-decreasing in $u$ and $v$. The goal is to show that 
\begin{align}\label{bol1}
\Phi(B(x),B(y),(y-x)) \geq 2B\Big(\frac{x+y}{2}\Big)
\end{align}
holds for all $0\leq x\leq y \leq 1$ and the choice of $B$ in \eqref{eq:B}. There will be six cases in total  to consider.

\begin{enumerate}[(i)]
    \item If $x,y \in [0,1/4]$.  Then $(x+y)/2 \in [0,1/4]$, so $B$ in (\ref{bol1}) coincides with $B_{1}$ and then (\ref{bol1}) follows from Lemma~\ref{help1}.
    \item If $x,y \in [3/4,1]$.  Then $(x+y)/2 \in [3/4,1]$, so $B$ in (\ref{bol1}) coincides with $B_{1}$. Denote $x=1-b$, $y = 1-a$. Then $0\leq a \leq b \leq 1/4$, and we have 
\begin{align*}
\Phi(B(x),B(y),(y-x)) & = \Phi(B_{1}(x),B_{1}(y),(y-x)) =\Phi(B_{1}(b),B_{1}(a),(b-a))\\
&\geq 2B_{1}((a+b)/2)=2B((x+y)/2),
\end{align*}
where the inequality follows by Lemma \ref{help2}.
\item If $x, y \in [1/4,3/4]$. Then $(x+y)/2 \in [1/4,3/4]$, so $B$ in (\ref{bol1}) coincides with $B_{2}$ and then \eqref{bol1} follows from Lemma~\ref{help3}. 
\item If $x \in [0,1/4], y \in [3/4,1]$. Then $(x+y)/2 \in [1/4,3/4]$. Therefore
\begin{align*}
\Phi(B(x),B(y),(y-x)) & \geq  \Phi(B_{2}(x),B_{2}(y),(y-x)) \\& \geq
2B_{2}((x+y)/2)=2B((x+y)/2),
\end{align*} 
where the second inequality follows by Lemma \ref{help3}.
\item If $x \in [0,1/4], y \in [1/4,3/4]$. Then $(x+y)/2 \in [1/8, 1/2]$. We distinguish two cases:
\begin{itemize}
    \item if $(x+y)/2 \in [0,1/4]$, then $y\leq 1/2$ and we have 
\begin{align*}
\Phi(B(x),B(y),(y-x)) & \geq \Phi(B_{1}(x),B_{1}(y),(y-x)) \\
& \geq
2B_{1}((x+y)/2) =  2B((x+y)/2),
\end{align*}
where in the second inequality we have used Lemma \ref{help1}.
    \item If $(x+y)/2 \in [1/4,3/4]$ then 
\begin{align*}
\Phi(B(x),B(y),(y-x)) & \geq \Phi(B_{2}(x),B_{2}(y),(y-x))\\
&\geq 2B_{2}((x+y)/2) =  2B((x+y)/2),
\end{align*}
where in the second inequality we have used Lemma \ref{help3}.
\end{itemize}
\item If $x \in [1/4,3/4], y \in [3/4,1]$. Then $(x+y)/2 \in [1/2, 7/8]$. We distinguish two cases:
\begin{itemize}
    \item If $(x+y)/2 \in  [1/4, 3/4]$ then we use the bound $B(y)\geq B_{2}(y)$ and proceed by Lemma~\ref{help3}, similarly to (iv), or the second item in (v).
    \item If  $(x+y)/2 \in  [3/4, 1]$, then $x \geq 1/2$. As in (ii), denote $x=1-b$, $y=1-a$. Note that $0 \leq a \leq 1/4 \leq b \leq 1/2$ and we have
\begin{align*}
\Phi(B(x),B(y),(y-x)) & \geq \Phi(B_{1}(x),B_{1}(y),(y-x))  = \Phi(B_{1}(b),B_{1}(a),(b-a)) \\
& \geq 2B_{1}((a+b)/2)  =2B((x+y)/2),
\end{align*}
where in the second inequality we have used Lemma \ref{help2}.
\end{itemize}
\end{enumerate}

Thus we have verified that \eqref{bol1} holds for the choice of $B$ in \eqref{eq:B}. Consequently, combining this with Lemma~\ref{th44} we obtain 
\[
\mathbb{E} h^{\beta_{0}}_{A} \geq B(\mu(A)) \geq \mu(A)^{*} (\log_{2}(1/\mu(A)^{*}))^{\beta_{0}}.
\]
This finishes the proof of the main inequality (\ref{iso2}).

To verify (\ref{iso1}), we apply H\"older's inequality to \eqref{iso2} for any $\alpha \in [0,1]$, and any $A \subset \{0,1\}^{n}$ with $\mu(A)\leq 1/2$. This yields,
\begin{align*}
(\mathbb{E} h_{A}^{\beta_{0}/\alpha})^{\alpha} \mu(A)^{1-\alpha} \geq (\mathbb{E} h_{A}^{\beta_{0}/\alpha})^{\alpha} \mu(\partial A)^{1-\alpha} \geq \mathbb{E} h_{A}^{\beta_{0}} \geq \mu(A) (\log_{2} (1/\mu(A)))^{\beta_{0}},
\end{align*}
where the first inequality uses the trivial bound $\mu(\partial A) \leq \mu(A)$. Thus we obtain $(\mathbb{E} h_{A}^{\beta_{0}/\alpha})^{\alpha} \mu(A)^{1-\alpha} \geq \mu(A) (\log_{2} (1/\mu(A)))^{\beta_{0}}$.  Dividing both sides of the inequality by $\mu(A)^{1-\alpha}$ and rising it to the power $1/\alpha$ we obtain (\ref{iso1}). This concludes the proof of Theorem~\ref{mth11}.

\subsection{Proof of Theorem \ref{thm:quadratic}: quadratic polynomial}\label{sec:quad}

The proof of Theorem \ref{thm:quadratic} will be a consequence of the upcoming Lemmas \ref{lemma:quad case1} and \ref{lemma:quad case 2}, when combined with Lemma \ref{th44}.

\begin{lemma}\label{lemma:quad case1}
Let $\beta \in [1/2,\log_2(3/2)]$. If $C_\beta=2 \sqrt{2^{\beta+1}-2}$ and $B(x)=C_\beta x (1-x)$, then the inequality
    \begin{align}\label{dav1}
     y-x+(2^{\beta}-1)B(y)
+B(x)\geq 2B\Big(\frac{y+x}{2}\Big)
    \end{align}
    holds 
    for all  $0 \leq y \leq 1$ and $0 \leq x \leq y -B(y)$.
\end{lemma}

\begin{proof}

Let $f(x,y)$ be the left-hand side of (\ref{dav1}) minus the right-hand side of (\ref{dav1}), that is, 
\begin{align}\label{eq:def f}
f(x,y) := C_{\beta}(2^{\beta}-2) y(1-y)-\frac{1}{2}(x-y)(C_{\beta}(x-y)+2).
\end{align}
The map $x \mapsto f(x,y)$ is concave, so in order to see that $f(x,y) \geq 0$ for $y \in [0,1]$ and $x \in [0, y -B(y)]$, it suffices to check the two inequalities: (i) $f(y-B(y),y)\geq 0$ and (ii) $f(0,y)\geq 0$. 
To verify (i) we have
\begin{align*}
f(y-B(y),y) = \frac{C_{\beta}}{2}y(1-y)(2^{\beta+1}-2-C^{2}_{\beta}y(1-y)).
\end{align*}
Notice that $2^{\beta+1}-2-C^{2}_{\beta}y(1-y) \geq 2^{\beta+1}-2-\frac{C^{2}_{\beta}}{4} = 0$ for all $y \in [0,1]$. To verify (ii) we have 
\begin{align*}
2f(0,y)/y=yC_{\beta}(3-2^{\beta+1})+2^{\beta+1}C_{\beta}-4C_{\beta}+2.
\end{align*}
Notice that $3-2^{\beta+1}\geq 0$ for all $\beta \in [1/2,\log_{2}(3/2)]$. Furthermore,  $y-B(y)\geq 0$ implies $\frac{C_{\beta}-1}{C_{\beta}}\leq y$. The value of $2f(0,y)/y$ at the point $y= \frac{C_{\beta}-1}{C_{\beta}}$ equals to $(\sqrt{2^{\beta+1}-2}-1)^{2} \geq 0$. 
\end{proof}

\begin{lemma}\label{lemma:quad case 2}
Let $\beta \in [1/2, \log_2(3/2)]$.  If $C_\beta=2 \sqrt{2^{\beta+1}-2}$ and $B(x)=C_\beta x (1-x)$, then the inequality
\begin{align}\label{dav23}
\big((y-x)^{1/\beta}+B(y)^{1/\beta})^{\beta}
+B(x)\geq 2B\Big(\frac{y+x}{2}\Big)
\end{align}
holds for all $0 \leq x \leq y \leq 1$ such that $x \in [y-B(y), y]$.
\end{lemma}

\begin{proof}
Let $g(x,y)$ be the left-hand side of (\ref{dav23}) minus the right-hand side of (\ref{dav23}), that is,
\begin{align*}
g(x,y)&:=\left[(C_{\beta}y(1-y))^{1/\beta}+(y-x)^{1/\beta}\right]^{\beta}-C_{\beta}y(1-y) - \frac{C_{\beta}}{2}(y-x)^{2}.
\end{align*}
%
We observe that $g(y,y)=0$ and, by the proof of Lemma~\ref{lemma:quad case1} we have $g(y-B(y),y)=f(y-B(y),y) \geq 0$ on $[0,1]$, where $f$ is as in \eqref{eq:def f}. Then, for a fixed $y\in[0,1]$, in order to prove that $g(x,y)\geq0$ for all $x\in[y-B(y),y]$, it  suffices to show the following two claims: 
\begin{enumerate}[(i)]
    \item the map $x \mapsto g(x,y)$ has at most one critical point on $(y-B(y),y)$;
    \item there exists $\epsilon>0$ such that $g(x,y)>0$ for all $x\in(y-\epsilon,y)$.
\end{enumerate}

We first show (i). We have that
\begin{equation}\label{eq:der g}
    \partial_{x}g(x,y)=-(y-x)^{\frac{1}{\beta}-1}\left[(C_{\beta}y(1-y))^{\frac{1}{\beta}}+(y-x)^{\frac{1}{\beta}}\right]^{\beta-1}+C_{\beta}(y-x).
\end{equation}
Let $A:=y-x$ and $B:=C_{\beta}y(1-y)$. Notice that $0\leq A\leq B$, and $B$ is independent of $x$. We have
$$
\partial_x g(x,y)=-A^{\frac{1}{\beta}-1}[B^{\frac{1}{\beta}}+A^{\frac{1}{\beta}}]^{\beta-1}+C_{\beta}A,
$$
so that $\partial_x g(x,y)=0$ is equivalent to  
\begin{equation}\label{eq:der g =0 quad}
B^{\frac{1}{\beta}}+A^{\frac{1}{\beta}}=C^{\frac{1}{\beta-1}}_{\beta}A^{\frac{2\beta-1}{\beta(\beta-1)}}.
\end{equation}
For $\beta\geq 1/2$, the left-hand side above increases in $A$, whereas the right-hand side decreases in $A$. Thus, \eqref{eq:der g =0 quad} can hold in at most one point, establishing (i). 

Item (ii) follows from noting that $\lim_{x \to y} \frac{\partial_x g(x,y)}{(y-x)^{\frac{1}{\beta}-1}}<0$ and $g(y,y)=0$.
\end{proof}

\subsection{Proof of Theorem \ref{thm:cubic}: cubic polynomial}\label{sec:cubic}
In this section we will be concerned with the polynomial 
\begin{align}\label{jd3}
B(x)=8x(1-x) \Big[ \Big(1-\frac{2^{\alpha+1}}{3}\Big) x + \Big(\frac{2^\alpha}{3}-\frac{1}{4}\Big) \Big].
\end{align}
 We remark that $B$ is chosen to 
satisfy $B(0)=B(1)=0$, $B(1/2)=1/2$ and  $B(1/4)=2^\alpha/4$. It will be convenient for later use to expand $B$ as 
\begin{equation}\label{eq:B cubic expanded}
    B(x)=8\Big(\frac{2^\alpha}{3}-\frac{1}{4}\Big)x  - 8 \Big(2^\alpha-\frac{5}{4}\Big)x^2 - 8 \Big(1-\frac{2^{\alpha+1}}{3}\Big) x^3.
\end{equation}

Theorem \ref{thm:cubic} follows from combining Lemma \ref{th44} and the upcoming Lemmas \ref{lemma:cubic case 1} and \ref{lemma:cubic case 2}.

\begin{lemma}\label{lemma:cubic case 1}
    Let $\alpha=1/2$ and  $\beta=0.53$. The inequality
    \[
     y-x+(2^{\beta}-1)B(y)
+B(x)\geq 2B\Big(\frac{y+x}{2}\Big).
    \]
    holds for all $0 \leq x \leq y -B(y)$, $0 \leq y \leq 1$. 
\end{lemma}

\begin{proof} 
We first observe that if $B(y)\leq y$, 
then $y \geq 1/2$. Indeed, a computation reveals that
\begin{equation}\label{eq:a_0-B(a_0)}
    y-B(y)= \frac{8(3-2^{\alpha +1})}{3}y\left(y-\frac{1}{2}\right)\left(y-\frac{2^{\alpha}-9/8}{2^{\alpha}-3/2}\right).
    \end{equation}
The claim now follows trivially since $\frac{2^{\alpha}-9/8}{2^{\alpha}-3/2} < 0$ for $\alpha=1/2 \leq \log_2(3/2)$. 

Let 
\begin{align}\label{ffdef}
    f(x,y):=(2^\beta-1)B(y) + (y-x) - 2 B\Big(\frac{x+y}{2}\Big) + B(x).
\end{align}
To verify Lemma~\ref{lemma:cubic case 1} it suffices to show the inequality $f(x,y) \geq 0$ for all $1/2 \leq y \leq 1$ and  all $0 \leq x \leq y- B(y)$. The inequality  will be implied by the following 3 claims: 
\begin{enumerate}[(i)]
    \item $f(0,y) \geq 0$ for all $y \in [1/2,1]$;
    \item $x \mapsto f(x,y)$ is concave on $[0,y-B(y)]$ for each fixed $y \in  [1/2,1]$;
    \item $f(y-B(y),y) \geq 0$ for all $y \in [1/2,1]$.
\end{enumerate}

To verify (i), notice that $\varphi(y):= f(0,y)/y$ is concave parabola as 
$$
3\varphi''(y) = \left(20-2^{\beta+4} \right)\left(3-2^{\alpha+1}\right)<0.
$$
On the other hand, $\varphi(1/2)=2^{\beta}-2^{\alpha}>0$ and $\varphi(1)=0.$ Thus $\varphi(y)\geq 0$ on $[1/2,1]$, and (i) follows.

To verify (ii) notice that 
$$
\partial^{2}_{xx}f(x,y) = -2(3-2^{\alpha+1})(6x-2y+2^{\alpha+1}+1)\leq 0 \quad \text{for all} \quad x,y \in [0,1].
$$


Next we verify (iii). Notice that $f(y-B(y),y) = B(y)G(y)$, where 
\begin{align*}
G(y) & =\Big(11-\frac{70}{9} \cdot 2^\alpha\Big) 2^{7}y^6 + (2^{\alpha+7} - 181) \frac{2^{6}}{3}y^{5} + \Big(1453-\frac{3082}{3}\cdot  2^\alpha \Big)\frac{2^{3}}{3}y^{4}\\
& -\Big(\frac{331}{3}-78 \cdot 2^{\alpha}\Big)2^{4}y^{3} +\Big(563-\frac{1192}{3} \cdot 2^\alpha \Big)\frac{2}{3}y^{2}-\Big(47-2^{\alpha+5}\Big)\frac{2}{3}y\\
&+2^{\beta} - 1
\end{align*}
is a polynomial of degree 6.  We have $G'''(y)\geq G'''(1/2)=(5\sqrt{2}-7)2^{7}  \geq 0$ on $[1/2,1]$ since the signs of all coefficients of the polynomial $G'''(y)$ are positive except of the free coefficient.  Also $G''(1/2)=(147-100\sqrt{2})\cdot 4/9> 0$. Thus $G(y)$ is convex on $[1/2,1]$. Therefore,  $G(y)\geq L(y) = G(y_{0})+G'(y_{0})(y-y_{0})$, where $y_{0}=0.631$. Since $G'(y_{0})=0.00036...>0$ and $L(1/2)=0.000067...>  0$  it follows that $G(y) \geq 0$ on $[1/2,1]$. 
\end{proof}

\begin{lemma}\label{lemma:cubic case 2}
Let $\alpha=0.5$ and $\beta=0.53$. The inequality
\[\big((y-x)^{1/\beta}+B(y)^{1/\beta})^{\beta}
+B(x)\geq 2B\Big(\frac{y+x}{2}\Big)
\]
holds for all $0 \leq x,  y \leq 1$ such that $x \in [\max\{y-B(y),0\}, y]$.
\end{lemma}

\begin{proof}
Let
\begin{equation*}
    g(x,y):= \Big( B(y)^{\frac{1}{\beta}} + (y-x)^{\frac{1}{\beta}} \Big)^\beta - 2B\Big(\frac{x+y}{2}\Big)+B(x).
\end{equation*}
The goal is to show  $g(x,y)\geq 0$ for all  $0 \leq y \leq 1$ and all $\max\{y-B(y),0\}\leq x \leq y$. It is clear that $g(y,y)=0$. We recall from the proof of Lemma~\ref{lemma:cubic case 1} that $\max\{y-B(y),0\}=y-B(y)$ if and only if $y \in [1/2,1]$. Therefore, notice that for $y \in [1/2,1]$ we have $g(y-B(y),y)=f(y-B(y),y) \geq 0$   due to positivity of the polynomial $G(y)$  on $[1/2,1]$ defined in Lemma~\ref{lemma:cubic case 1}; here $f$ is as in \eqref{ffdef}. If $y \in [0,1/2]$, then 
\begin{align*}
g(0,y) &= (B(y)^{1/\beta}+y^{1/\beta})^{\beta}-2B(y/2) \geq (B(y)^{2}+y^{2})^{1/2}-2B(y/2).
\end{align*}
Next, we have
\begin{align*}
&B(y)^{2}+y^{2}-(2B(y/2))^{2} =\\
&\frac{17-12\sqrt{2}}{3}\cdot  y^{2}(1-2y)(1-y) (10y^{2}+y(29+28\sqrt{2})+51+36\sqrt{2})\geq 0
\end{align*}
for all $y \in [0,1/2]$. So we obtain  $g(\max\{y-B(y),0\},y)\geq 0$ for all $y \in [0,1]$. 

Thus to prove Lemma~\ref{lemma:cubic case 2} it suffices to obtain the following two claims for each fixed  $y \in [0,1]$: 
\begin{enumerate}[(i)]
    \item the map $x \mapsto g(x,y)$ has at most one critical point on $(0,y)$;
    \item there exists $\varepsilon>0$ such that $\partial_x g(x,y)<0$ for all $x \in (y-\varepsilon,y)$. 
\end{enumerate}

For the claim (ii), an immediate computation shows
\begin{align*}
    \partial_x g(x,y) & = - \Big( B(y)^{\frac{1}{\beta}} + (y-x)^{\frac{1}{\beta}} \Big)^{\beta-1} (y-x)^{\frac{1}{\beta}-1} - B'\Big( \frac{x+y}{2} \Big) + B'(x).
\end{align*}
Therefore, since $\frac{1}{\beta}-1<1$ it follows that $\lim_{x\to y} \frac{\partial_{x}g(x,y)}{(y-x)^{\frac{1}{\beta}-1}}>0$.

To verify the claim (i),  notice that 
\begin{align*}
&(y-x)^{1-\frac{1}{\beta}}\partial_x g(x,y) =- \Big( B(y)^{\frac{1}{\beta}} + (y-x)^{\frac{1}{\beta}} \Big)^{\beta-1} +\\
&2(3-2^{3/2})(y-x)^{2-\frac{1}{\beta}}(3x+y+2^{3/2}+1).
\end{align*}
It suffices to show that for each $y \in [0,1]$ the map
\begin{align*}
\psi(x) :=-(B(y)^{\frac{1}{\beta}} + (y-x)^{\frac{1}{\beta}})+[2(3-2^{3/2})(y-x)^{2-\frac{1}{\beta}}(3x+y+2^{3/2}+1)]^{\frac{1}{\beta-1}}
\end{align*}
 satisfies $\psi'>0$ on $(0,y)$. We have 
\begin{align}
&\frac{\psi'(x)(y-x)^{-\frac{2\beta-1}{\beta(\beta-1)}+1}}{(2(3-2\sqrt{2})(3x+y+2^{3/2}+1))^{\frac{1}{\beta-1}}} 
=\frac{((y-x)(3x+y+2^{3/2}+1))^{\frac{1}{1-\beta}}}{\beta(2(3-2\sqrt{2}))^{\frac{1}{\beta-1}}}+\frac{2\beta-1}{\beta(1-\beta)} \nonumber\\
&-\frac{3(y-x)}{(1-\beta)(3x+y+2^{3/2}+1)}. \label{marjvena}
\end{align}
Set $x=y-t$ in (\ref{marjvena}) where $y \in [t,1]$ and $t \in (0,1)$. Then the right-hand side of (\ref{marjvena}) is increasing in $y$ for each fixed $t$,  and for $y=t$ it takes the form  
\begin{align*}
\frac{(t(t+2^{3/2}+1))^{\frac{1}{1-\beta}}}{\beta(2(3-2\sqrt{2}))^{\frac{1}{\beta-1}}}+\frac{2\beta-1}{\beta(1-\beta)} -\frac{3t}{(1-\beta)(t+2^{3/2}+1)}=:\varphi(t).
\end{align*}
Clearly  $\varphi(t)$ is convex on $[0,1]$ as it is the sum of two convex functions. Therefore $\varphi(t)\geq L(t):=\varphi(0.22)+\varphi'(0.22)(t-0.22)$. Since $\varphi'(0.22)=0.053...>0$ and $L(0)=0.033...>0$  the lemma follows. 
\end{proof}

We finish this section discussing how Theorem \ref{thm:cubic} (or Corollary \ref{cor:quad large measure}) implies Corollary \ref{cor:sep cube}. This is already contained in \cite[Corollary 3.2]{KP}, but we include it here for completeness.

\begin{proof}[Proof of Corollary \ref{cor:sep cube}]
Let $(A,B,W)$ be a partition of $\{0,1\}^n$ with $\mu(A)=1/2$. Assume that $\mathbb{E} h_A^\beta  \geq P(\mu(A))$ holds for all such $A$. Then we have 
\begin{equation*}\label{eq:thm implies cor}
P(\mu(A))\le \mathbb{E} h_{B\cup W}^\beta =
\mathbb{E} (h_{B\cup W}^\beta \bbone_{B}) + \mathbb{E} (h_{B\cup W}^\beta \bbone_{W} )  \le \frac{|\nabla(A,B)|}{2^n} +n^\beta\mu(W),
\end{equation*}
where in the last inequality we have used the bound $h_{B \cup W} \leq n$. 
The desired result then follows if $P(1/2)=1/2$, which is the case for the function on the right-hand side of \eqref{eq: cubic}. 
\end{proof}

\subsection{Proof of Theorems~\ref{vectortal} and \ref{nstabilityt}}
We start from an  identity in \cite{PRS}  stating that 
\begin{align}\label{ihv1}
    -\frac{d}{dt} e^{-t\Delta }f(x) = \frac{1}{\sqrt{e^{2t}-1}} \mathbb{E}_{\xi} \sum_{j=1}^{n} \delta_{j}(t) D_{j} f(\xi(t) x),
\end{align}
where $x \xi(t) = (x_{1}\xi_{1}(t), \ldots, x_{n} \xi_{n}(t))$, and $\xi_{i}(t)$ are i.i.d. r.v. with 
\begin{align*}
\mathbb{P}(\xi_{j}(t)=\pm 1) = \frac{1\pm e^{-t}}{2}, \quad \delta_{j}(t) = \frac{\xi_{j}(t) -\mathbb{E} \xi_{j}(t)}{\sqrt{\mathrm{Var}(\xi_{j}(t))}}.
\end{align*}
The identity is easily verified for  Walsh functions $x^{S}$, and by linearity it follows for all functions $f$. 
By \eqref{ihv1} and  the fundamental theorem of calculus we have
\begin{align*}
\| f-e^{-t\Delta}f\|_{p} & \stackrel{(\ref{ihv1})}{\lesssim} \int_{0}^{t}\Big(\mathbb{E} \Big\| \sum_{j=1}^{n} \delta_{j}(s) D_{j} f(x) \Big\|^{p}\Big)^{1/p} 
\frac{ds}{\sqrt{e^{2s}-1}} \\
& \,\,\,\, \lesssim C_{q}(X) \sqrt{q}\|Df\|_{p} \int_{0}^{t}(1-e^{-2s})^{\frac{1}{q}-1}ds \\ 
& \,\,\,\, \lesssim C_{q}(X) q^{3/2} (1-e^{-2t})^{\frac{1}{q}} \|Df\|_{p}, 
\end{align*}
where we also  used Theorem 4.1 in \cite{PRS} in the second inequality. This concludes the proof of Theorem \ref{nstabilityt}.

For Theorem \ref{vectortal}, we note that since $\| e^{-t\Delta} f\|_{\infty} \leq \|f\|_{\infty}\leq 1$ the above inequality and Minkowski's inequality imply that
\begin{align}\label{nstability}
\|f\|_{2}^{2/p} -\|e^{-t\Delta}f\|_{2}^{2/p}  \lesssim C_{q}(X) q^{3/2} (1-e^{-2t})^{\frac{1}{q}} \|Df\|_{p}.
\end{align}
Denote $\varepsilon  = 1-e^{-2t}$. By hypercontractivity~\cite{Bon1} and the H\"older inequality  we have $\| e^{-t\Delta} f\|_{2}  \leq \|f\|_{2}\left(\frac{\|f\|_{1}}{\|f\|_{2}}\right)^{\frac{\varepsilon}{2-\varepsilon}}.$ Thus  
\begin{align}\label{diffp1}
 \frac{1-\left(\frac{\|f\|_{1}}{\|f\|_{2}}\right)^{\frac{2\varepsilon}{p(2-\varepsilon)}}}{\varepsilon^{1/q}} \lesssim C_{q}(X) q^{3/2} \frac{\|Df\|_{p}}{\|f\|_{2}^{2/p}}.
\end{align}
If $\frac{\|f\|_{2}}{\|f\|_{1}}<10$, then the theorem follows by taking $t\to \infty$ in (\ref{nstability}). Otherwise, choosing $\varepsilon = \frac{1}{\log \frac{e \|f\|_{2}}{\|f\|_{1}}}$ in (\ref{diffp1}) proves the theorem.

\section{Appendix}\label{sec:nec}

For completeness, we include the argument that shows the necessity of the condition $\beta \geq 1/2$ in the isoperimetric inequalities \eqref{conj1} and \eqref{eq:conj n beta}. In both cases, the counterexamples are given by Hamming balls.
Let $a \in \{0,1\}^n$ and $r>0$. Let $B(a,r)$ denote the Hamming ball of center $a$ and radius $r$, defined with respect to the $\ell^1$-distance in $\{0,1\}^n$. Note that the number of elements of $B(a,r)$ is
$\sum_{i=0}^{r^*} \binom{n}{i}$, 
where $r^*:=\min\{\lfloor r \rfloor, n\}$. If $r \in \mathbb{N}$, $r \leq n$, the sphere $S(a,r)$ has $\binom{n}{r}$ elements.

We start with the claim in \eqref{eq:beta>=1/2}. Let $n$ be a positive even integer, $a \in \{0,1\}^n$ arbitrary, and consider $A=B(a,n/2)$. Clearly, $\supp(h_A)=S(a,n/2)$. Since  $h_A(x)=n/2$ for $x \in S(0,n/2)$, and $h_A(x)=0$ otherwise,  we have
\begin{equation*}
    \mathbb{E} h^\beta_A =  (n/2)^\beta \mu (\supp h_A) = (n/2)^\beta \mu (S(0,n/2))=(n/2)^\beta \frac{\binom{n}{n/2}}{2^n}.
\end{equation*}
Thus,  Stirling's approximation gives $\lim_{n \to \infty}\mathbb{E} h^\beta_A=0$
for $\beta < 1/2$. 


To verify the necessity  of the condition $\beta\geq 1/2$ in  \eqref{eq:conj n beta} we let $n$ be a positive odd integer,  and let $a \in \{0,1\}^n$ be arbitrary. Consider the Hamming ball $A=B(a,(n-1)/2)$, which has measure $\mu(A)=2^{-n}\sum_{i=0}^{(n-1)/2} \binom{n}{i}=1/2$, and let $W$ be the boundary of $A$, and let $B:=\{0,1\}^d\setminus{(A\cup W)}$. Clearly $|\nabla(A,B)|=0$. Then the inequality \eqref{eq:conj n beta}  reduces to 
\begin{equation}\label{eq:sep cube ball}
K|W|n^{\beta}\geq 2^{n-1}.
\end{equation}
Since $|W|=\binom{n}{(n-1)/2}$, it follows from Stirling's approximation that the condition $\beta \geq 1/2$ is necessary for the inequality $\liminf_{n\to \infty} K \frac{|W| n^{\beta}}{2^{n-1}}\geq 1$ to hold.


\end{document}